\documentclass[reqno,12pt]{article}
\usepackage{amssymb,amsthm,amsmath,amsfonts}
\usepackage{epsfig}

\theoremstyle{plain}
\begingroup

\newtheorem{thm}{Theorem}[section]
\newtheorem{theorem}{Theorem}[section]

\newtheorem{corollary}[thm]{Corollary}

\newtheorem{proposition}[thm]{Proposition}

\endgroup
\theoremstyle{definition}
\newtheorem{defn}{Definition}[section]

\newtheorem{remark}[defn]{Remark}

\newtheorem{example}[defn]{Example}

\theoremstyle{remark}

\numberwithin{equation}{section}
\numberwithin{figure}{section}

\newcommand{\nc}{\newcommand}
\nc{\G}{{\Gamma}}
\nc{\BC}{{\mathbb C}}
\nc{\BQ}{{\mathbb Q}}
\nc{\BR}{{\mathbb R}}
\nc{\BZ}{{\mathbb Z}}
\nc{\BP}{{\mathbb P}}
\nc{\BN}{{\mathbb N}}
\nc{\BM}{{\mathbb M}}
\nc{\fH}{{\mathfrak{H}}}
\nc{\vp}{{\varepsilon}}\nc{\dpar}{{\partial}}\nc{\al}{{\alpha}}
\nc{\PSL}{{\mbox{PSL}_2(\BR)}}
\nc{\PS}{{\mbox{PSL}_2(\BZ)}}
\nc{\SL}{{\mbox{SL}_2(\BZ)}}

\DeclareMathOperator{\re}{Re} \DeclareMathOperator{\im}{Im}

\DeclareMathOperator{\supp}{\mathrm{supp}}

\DeclareMathOperator{\Res}{\mathcal{R}}
\DeclareMathOperator*{\res}{\mathrm{Res}}

\def\I{\mathrm{i}}

\def\scal#1#2{\langle #1, #2\rangle }

\def\Log{{\rm Log\,}}
\def\D{{\mathbb D}}
\def\R{{\mathbb R}}
\def\C{{\mathbb C}}
\def\P{{\mathbb P}}
\def\Z{{\mathbb Z}}

\begin{document}

\title{Exponential transforms, resultants
and moments}
\author{Bj\"orn Gustafsson\textsuperscript{1} }

\date{December 3, 2012}

\maketitle

\begin{abstract}
We give an overview of some recent developments concerning harmonic and other moments of
plane domains, their relationship to the Cauchy and exponential transforms,
and to the meromorphic resultant and elimination function. The paper also connects to certain topics in mathematical physics, for
example domain deformations generated by harmonic gradients (Laplacian growth) and related integrable structures.
\end{abstract}



\footnotetext[1] {Department of Mathematics, KTH, 100 44,
Stockholm, Sweden.\\ Email: \tt{gbjorn@kth.se}}


\noindent {\it Acknowledgements:} This work has been performed within the framework of
the European Science Foundation Research Networking Programme HCAA,
and has been supported by the Swedish Research Council and the G\"oran Gustafsson Foundation,


\section{Introduction}

In the present article we will focus on developments concerning harmonic and other moments
of a domain in the complex plane, the exponential transform and, in the case of a quadrature domain,
the relation between these objects and the resultant and elimination function.
Much of the material is based on joint work with Vladimir Tkachev, Mihai Putinar, Ahmed Sebbar
and is in addition inspired by ideas of Mark Mineev-Weinstein, Paul Wiegmann, Anton Zabrodin and
others.

The organization of the paper is as follows. Section~\ref{sec:Cauchy} gives the basic definitions of
Cauchy and exponential transforms, including extended versions in four complex variables. In Section~\ref{sec:moments}
it is shown how these generate harmonic and exponential moments. In the case of quadrature domains
(equivalently, algebraic domains or finitely determined domains), there are strong algebraic relationships between the various
transforms and moments, and this is exposed in Section~\ref{sec:finitely}.
Most of the material in Sections~\ref{sec:Cauchy}--\ref{sec:finitely} is by now relatively classical, being based on developments in the period 1970-2000 within complex analysis and operator theory. For example, the exponential transform came out as a byproduct from the theory
of hyponormal operators.

Central for the paper is Section~\ref{sec:resultant} on the meromorphic resultant, which was introduced in \cite{Gustafsson-Tkachev-2009}.
One main message is that, in the case of quadrature domains, the elimination function, which is defined by means of the
meromorphic resultant (and hence is a purely algebraic object), turns out to be the same as the exponential transform (which is an analytic object).
In Section~\ref{sec:potential}, some potential theoretic interpretations of the resultant are given.

In the last three sections,  Section~\ref{sec:moment coordinates}--\ref{sec:hamiltonians}, we
connect the previously discussed material to some quite exciting and relatively recent developments in mathematical physics.
It much concerns deformations (or variations) of domains when the harmonic moments are used as coordinates
(in the simply connected case). Such deformations fit into the framework of Laplacian growth processes,
or moving boundaries for Hele-Shaw flows. More exactly,
they can naturally be thought of as evolutions driven by harmonic gradients.
We mostly restrict to finite dimensional subclasses of domains, namely polynomial (of a fixed degree)
conformal images of the unit disk, this in order to keep the presentation transparent and rigorous.
A nice feature then is that the Jacobian determinant for the transition between harmonic moment coordinates and the coordinates provided by the
coefficients of the polynomial mapping functions can be made fully explicit, (\ref{jacobian}), in terms of a resultant involving the mapping
function. This result is due to O.~Kuznetsova and V.~Tkachev \cite{Kuznetsova-Tkachev-2004},
\cite{Tkachev-2005}.

The further results concern the string equation (\ref{string}) (which is equivalent to the Polubarinova-Galin equation for
a Hele-shaw blob), integrability properties of the harmonic moments (\ref{M_k}),
and a corresponding prepotential (\ref{prepotential}), which is the logarithm of a $\tau$-function and which
in some sense explains the mentioned integrability properties. Finally, we
give Hamiltonian formulations of the evolution equation for the conformal map under
variation of the moments (\ref{fH}). All these matters are due to M.~Mineev-Weinstein, P.~Wiegmann, A.~Zabrodin
and others, see for example
\cite{Wiegmann-Zabrodin-2000}, \cite{Mineev-Zabrodin-2001}, \cite{Marshakov-Wiegmann-Zabrodin-2002},
\cite{Krichever-Marshakov-Zabrodin-2005},  but we try to
explain everything in our own words and using the particular settings of the present paper.

\subsection{List of notations}\label{sec:notations}

Below is a list of general notations that will be used.

\begin{itemize}

\item $\mathbb{D}=\{\zeta\in \mathbb{C}:|\zeta|< 1\}$,
$\mathbb{D}(a,r)=\{\zeta\in \mathbb{C}:|\zeta-a|< r\}$.

\item $\P=\C\cup\{\infty\}$, the Riemann sphere.

\item $dm= dm(z)=dx\wedge dy=\frac{1}{2\I}d\bar{z}\wedge dz$.
($z=x+\I y$).

\item $f^{*}(\zeta)=\overline{f(1/\bar{\zeta})}$.

\end{itemize}\par



\section{The Cauchy and exponential transforms}\label{sec:Cauchy}

The {\it logarithmic potential} of a measure $\mu$ with compact support in $\C$ is
\begin{equation}\label{potential}
U_\mu(z)=\int \log\frac{1}{|z-\zeta|} d\mu(\zeta),
\end{equation}
satisfying $-\Delta U_\mu=2\pi\mu$ in the sense of distributions. The gradient of $U_\mu$ can,
modulo a constant factor and a complex conjugation, be identified with the {\it Cauchy transform}
$$
C_\mu (z)=-\frac{2}{\pi}\frac{\partial U_\mu(z)}{\partial z}=-\frac{1}{\pi}\int \frac{d\mu (\zeta)}{\zeta-z},
$$
which may be naturally viewed as a differential: $C_\mu(z)dz$. From these potentials, or fields, $\mu$ can be recovered by
\begin{equation}\label{dbarCauchy}
\mu=\frac{\partial C_\mu}{\partial \bar z} =-\frac{1}{2\pi}\Delta U_\mu.
\end{equation}

We shall mainly deal with measures of the form $\mu=\chi_\Omega m$, where $m$ denotes
Lebesgue measure in $\C$ and $\Omega\subset\C$ is a bounded open set. Then we write $U_\Omega$, $C_\Omega$
in place of $U_\mu$ and $C_\mu$. On writing $C_\Omega$ as
$$
C_\Omega (z)=\frac{1}{2\pi\I}\int_\Omega \frac{d\zeta \wedge d\bar \zeta}{\zeta-z}=\frac{1}{2\pi\I}\int_\Omega \frac{d\zeta }{\zeta-z}
\wedge d\bar \zeta.
$$
one is naturally led to consider the more symmetric ``double Cauchy transform'',
\begin{equation}\label{doubleCauchy}
C_\Omega (z,w)=\frac{1}{2\pi\I}\int_\Omega\frac{d\zeta }{\zeta-z} \wedge \frac{d\bar \zeta}{\bar \zeta-\bar w},
\end{equation}
which is much richer than the original transform. It may be viewed as a double differential: $C_\mu(z,w)\,dz d\bar w$.
After exponentiation it gives the by now quite well
studied \cite{Carey-Pincus-1974}, \cite{Putinar-1994}, \cite{Putinar-1996},
\cite{Putinar-1998}, \cite{Gustafsson-Putinar-1998} {\it exponential transform} of $\Omega$:
$$
E_\Omega(z,w)=\exp {C_\Omega (z,w)}.
$$
The original Cauchy transform reappears when specializing one variable at infinity:
$$
C_\Omega(z)=\res_{w=\infty} C_\Omega (z,w)\,d\bar w =-\lim_{w\to \infty} {\bar w} C_\Omega(z,w).
$$
As a substitute for (\ref{dbarCauchy}) we have, for the double transform,
$$
\frac{\partial^2 C_\Omega(z,w)}{\partial \bar z\partial w} =-{\pi}\delta(z-w)\chi_\Omega(z)\chi_\Omega(w).
$$

Even the double Cauchy transform is not fully complete. It
contains the Cauchy kernel $\frac{d\zeta}{\zeta-z}$, which is a
meromorphic differential on the Riemann sphere with a pole at
$\zeta=z$, but it has also a pole at $\zeta=\infty$. It is natural
to make the latter pole visible and movable. This has the
additional advantage that one can avoid the two Cauchy kernels which
appear in the definitions of the double Cauchy transform, and the
exponential transform, to have coinciding poles (namely at infinity).
Thus one arrives naturally at the extended (or four variable) Cauchy and exponential
transforms:
\begin{equation}\label{extendedCauchy}
C_\Omega(z,w;a,b)=\frac{1}{2\pi\I}\int_\Omega(\frac{d\zeta }{\zeta-z} -\frac{d\zeta }{\zeta-a})
\wedge (\frac{d\bar \zeta}{\bar \zeta-\bar w}-\frac{d\bar \zeta}{\bar \zeta-\bar b}),
\end{equation}
\begin{equation}\label{extendedexp}
E_\Omega(z,w;a,b)=\exp { C_\Omega(z,w;a,b)}=\frac{E_\Omega(z,w)E_\Omega(a,b)}{E_\Omega(z,b)E_\Omega(a,w)}.
\end{equation}
Clearly $C_\Omega(z,w)=C_\Omega(z,w;\infty,\infty)$, $E_\Omega(z,w)=E_\Omega(z,w;\infty,\infty)$.
If the points $z$, $w$, $a$, $b$ are taken to be all distinct, then
both transforms are well defined and finite for any open set $\Omega$ in the Riemann sphere $\P$.
For example, $E_\P(z,w;a,b)$ turns out to be the modulus
squared of the cross-ratio:
\begin{equation}\label{crossratio}
E_\P(z,w;a,b)=|(z:a:w:b)|^2,
\end{equation}
where, according to the classical definition \cite{Ahlfors-1966},
\begin{equation}\label{cross}
(z:a:w:b)=\frac{(z-w)(a-b)}{(z-b)(a-w)}
\end{equation}
with the variables in this order.
To prove (\ref{crossratio}), one may use the formula for the two-variable exponential transform for the disk $\Omega=\D(0,R)$
when both variables are inside, namely (see \cite{Gustafsson-Putinar-1998})
$$
E_{\D(0,R)}(z,w)=\frac{|z-w|^2}{R^2-z\bar w} \quad (z,w\in \D(0,R)),
$$
insert this into the last member of (\ref{extendedexp}) and let $R\to\infty$. A different proof of (\ref{crossratio})
will be indicated in Section~\ref{sec:potential}.

We record also the formula for the exponential transform for an arbitrary disk $\D(a,R)$ when
both variables are outside ($z,w\in\C\setminus\D(a,R)$):
\begin{equation}\label{Edisk}
E_{\D(a,R)}(z,w)=1-\frac{R^2}{(z-a)(\bar w-\bar a)}.
\end{equation}
See again \cite{Gustafsson-Putinar-1998}.

From (\ref{crossratio}) it is immediate that for any domain $\Omega\subset\P$,
\begin{equation}\label{EEcrossratio}
E_\Omega(z,w;a,b)E_{\P\setminus\Omega}(z,w;a,b)=|(z:a:w:b)|^2
\end{equation}
(clearly $E_D$ makes sense even if the set $D$ is not open).
The two-variable exponential transform $E_{\P\setminus\Omega}(z,w)$ is identically zero if $\Omega$ is bounded, but
there is still a counterpart of (\ref{EEcrossratio}) in this case. It is
\begin{equation}\label{EHcrossratio}
E_\Omega(z,w)\cdot\frac{1}{H_{\Omega}(z,w)}=|z-w|^2,
\end{equation}
where $H_\Omega(z,w)$ is the {\em interior exponential transform}, which is a renormalized version of one over the
exponential transform of the exterior domain, more precisely
$$
H_\Omega(z,w)=\lim_{R\to\infty}\frac{1}{R^2E_{\D(0,R)\setminus\Omega}}.
$$
We consider this function only for $z,w\in\Omega$, and then it is analytic in $z$ and antianalytic in $w$.

When the integral in the definition of the exponential transform is transformed into a boundary integral it turns out that
the formulas for $E_\Omega(z,w)$ (with the variables outside) and $H_\Omega(z,w)$ become exactly the same:
\begin{align*}
E_\Omega(z,w)&= \exp\big[\frac{1}{\I\pi}\int_{\partial\Omega}\log|\zeta-z|\,d\overline{\log(\zeta-w)}\big] \quad (z,w \in \C\setminus\Omega),\\
H_\Omega(z,w)&= \exp\big[\frac{1}{\I\pi}\int_{\partial\Omega}\log|\zeta-z|\,d\overline{\log(\zeta-w)}\big]  \quad (z,w \in\Omega).
\end{align*}
Here one can easily identify the modulus and the angular part, for example for $E_\Omega(z,w)$:
\begin{align*}
|E_\Omega(z,w)|&= \exp\big[-\frac{1}{\pi}\int_{\partial\Omega}\log|\zeta-z|\,d{\arg(\zeta-w)}\big],\\
\arg E_\Omega(z,w)
&= -\frac{1}{\pi}\int_{\partial\Omega}\log|\zeta-z|\,d{\log|\zeta-w|}
\end{align*}
($z,w \in \C\setminus\Omega$).
These formulas open up for geometrical interpretations.

The exponential transform originally arose as a side product in the theory of hyponormal operators on a Hilbert space,
see \cite{Carey-Pincus-1974}, \cite{Martin-Putinar-1989}, \cite{Pincus-Xia-Xia-1984}, and the inner product in the Hilbert space then automatically forces the transform
to have certain positive definiteness properties, namely the two variable transforms $1/E_\Omega$ and $H_\Omega$
are positive definite (as kernels) in all $\C$, and $1-E_\Omega$ is positive semidefinite outside $\Omega$.  See \cite{Gustafsson-Putinar-2005}
for direct proofs. For the four variable exponential transform we similarly have, for example, that $1/E$ is positive definite
in the sense that
$$
\sum_{k,j} \frac{\lambda_k\overline{\lambda_j}}{E_\Omega (z_k, z_j;a_k, a_j)}\geq 0
$$
for any finite tuples $(z_j,a_j)\in \C^2$, $\lambda_j\in\C$, and with strict inequality unless all $\lambda_j$ are zero.

Aside from operator theory, the ideas of the exponential transform appear implicitly as a tool in the theory of boundaries
of analytic varieties, see \cite{Harvey-Lawson-1975}, \cite{Alexander-Wermer-1998} (Chapter 19).


\section{Harmonic, complex and exponential moments}\label{sec:moments}

The Cauchy and exponential transforms are generating functions for series of
moments of a bounded domain  $\Omega\subset\C$. Specifically we have the {\it complex moments}
$M_{kj}=M_{kj}(\Omega)$,
$$
M_{kj}=\frac{1}{\pi} \int_\Omega z^k \bar{z}^j \,dm(z)=\frac{1}{2\pi\I} \int_\Omega z^k \bar{z}^j \, d\bar{z}
d{z} \quad (k,j\geq 0),
$$
which in view of the Weierstrass approximation theorem completely determine $\Omega$
up to nullsets, and the more restricted {\it harmonic moments} $M_k=M_k(\Omega)$:
\begin{equation}\label{harmonicmoments}
M_k=M_{k0}=\frac{1}{\pi} \int_\Omega z^k  \,dm(z)=\frac{1}{2\pi\I} \int_\Omega z^k  \, d\bar{z} dz
\quad (k\geq 0).
\end{equation}
The latter do not completely determine the domain, not even in the simply connected case
(see \cite{Sakai-1978}). However, for simply connected domains with analytic boundaries they are
at least sensible for local variations of the domain, {\it i.e.}, we have a local
one-to-one correspondence
\begin{equation}\label{correspondence}
\Omega \leftrightarrow (M_0, M_1, M_2,\dots).
\end{equation}
A precise statement in this respect is the following: there exists a compact subset $K\subset\Omega$
such that any Jordan domain $D\supset K$ with $M_k(D)=M_k(\Omega)$ for all $k\geq 0$ necessarily
agrees with $\Omega$. See \cite{Gustafsson-1990}, Corollary~3.10, for a proof.

It is actually not {\it a priori} obvious that the moments $M_k$ ($k\geq 0$) are independent of each other
(for example, the complex moments $M_{kj}$ are definitely not independent, given that they are moments
of domains), but as will be discussed in Sections~\ref{sec:moment coordinates} and \ref{sec:deformation},
it is indeed possible to define deformations of $\Omega$ which change one of the $M_k$ without changing
the other ones. See also \cite{Cherednichenko-1996} and \cite{Ross-Witt-Nystrom-2012} for further discussions
and results.

The generating properties for the harmonic and complex  moments are
\begin{equation}\label{cauchy}
C_\Omega (z)=\sum_{k=0}^\infty \frac{M_k}{z^{k+1}},
\end{equation}
$$
C_\Omega (z,w)=\sum_{k,j\geq 0} \frac{M_{kj}}{z^{k+1}\bar{w}^{j+1}}.
$$
Finally we introduce the {\it exponential moments} $B_{kj}=B_{kj}(\Omega)$ having
$1-E(z,w)$ as generating function, hence being defined by
\begin{equation}\label{oneminusE}
1-\exp[-\sum_{k,j\geq 0} \frac{M_{kj}}{z^{k+1}\bar{w}^{j+1}}]  =\sum_{k,j\geq 0} \frac{B_{kj}}{z^{k+1}\bar{w}^{j+1}}.
\end{equation}
The positive semidefiniteness of $1-E(z,w)$ implies that also the infinite matrix $(B_{kj})$ is positive semidefinite.
The extended (four variable) exponential transform similarly generate some four index moments, see \cite{Putinar-1988},
\cite{Putinar-1990}, \cite{Putinar-1996}, \cite{Putinar-1998}, \cite{Martin-Putinar-1989}.

In the remaining part of this section we shall assume that $\Omega$ is a simply connected domain with $0\in\Omega$
and such that $\partial\Omega$ is an analytic curve (without singularities).
By writing (\ref{harmonicmoments}) as a boundary integral,
\begin{equation}\label{Mkboundary}
M_k=\frac{1}{2\pi\I} \int_{\partial\Omega} z^k  \, \bar{z} dz,
\end{equation}
the harmonic moments become meaningful also for $k<0$, and these can be considered as functions of the moments for $k\geq
0$, as far as local variations are concerned. Thus we may write
$$
M_{-k}=M_{-k}(\Omega)=M_{-k}(M_0, M_1, M_2,\dots)
$$
for $k>0$.

In a series of papers by I.~Krichever, A.~Marshakov,
M.~Mineev-Weinstein, P.~Wiegmann and A.~Zabrodin, for example
\cite{Wiegmann-Zabrodin-2000}, \cite{Mineev-Zabrodin-2001}, \cite{Marshakov-Wiegmann-Zabrodin-2002},
\cite{Krichever-Marshakov-Zabrodin-2005}, the so extended moment sequence has been
shown to enjoy remarkable integrability properties, for example (in
present notation)
\begin{equation}\label{M_k}
\frac{1}{k}\frac{\partial M_{-k}}{\partial M_j}
=\frac{1}{j}\frac{\partial M_{-j}}{\partial M_k}\quad (k,j\geq 1).
\end{equation}
This can be explained in terms of the presence of a
prepotential $\mathcal{F}(\Omega)$ such that
\begin{equation}\label{dprepotential}
\frac{1}{k}M_{-k}(\Omega) = \frac{\partial \mathcal{F}(\Omega)}{\partial M_k}.
\end{equation}
In fact, for any sufficiently large $R>0$ (it is enough that $\overline{\Omega}\subset\D(0,R)$),
the energy functional
\begin{equation}\label{prepotential}
\mathcal{F}(\Omega)=\mathcal{F}_R(\Omega)=
\frac{1}{\pi^2}\int_{\D(0,R)\setminus\Omega}\int_{\D(0,R)\setminus\Omega} \log
\frac{1}{|z-\zeta|} dm(z) dm({\zeta})
\end{equation}
serves as a prepotential.
The exponential of $\mathcal{F}(\Omega)$ can be identified with a certain
$\tau$-function which appears as a partition function in mathematical models in
statistical mechanics, see \cite{Kostov-Krichever-Mineev-Wiegmann-Zabrodin-2001}.
An indication of the proof of (\ref{dprepotential}), and hence (\ref{M_k}), will be given in Section~\ref{sec:deformation}.

\begin{remark}
The dependence of $M_{-k}$ and $\mathcal{F}(\Omega)$ on the $M_j$ is actually not analytic, so the above partial
derivatives are really Wirtinger derivatives:
$$
\frac{\partial}{\partial M_j}=\frac{1}{2}(\frac{\partial}{\partial \re M_j}-\I \frac{\partial}{\partial\im  M_j}),
$$
$$
\frac{\partial}{\partial \bar{M}_j}=\frac{1}{2}(\frac{\partial}{\partial \re M_j}+\I \frac{\partial}{\partial\im  M_j}).
$$
\end{remark}

On extending the power series (\ref{cauchy}) for  the Cauchy transform to a full Laurent series by means of the negative
moments one gets, at least formally, the {\em Schwarz function} of $\partial\Omega$:
\begin{equation}\label{schwarz}
S (z)=\sum_{k=-\infty}^\infty \frac{M_k}{z^{k+1}}.
\end{equation}
Even though the full series here need not converge anywhere it can be given the following precise meaning: the negative part of the series defines a germ of an analytic function at the origin and the positive part a germ of an analytic function at infinity.
When $\partial\Omega$ is analytic the domains of analyticity of these two analytic functions overlap, and the overlap region contains
$\partial\Omega$. Thus $S(z)$, as written above, makes sense and is analytic in a neighborhood of $\partial\Omega$.
Moreover, it satisfies
\begin{equation}\label{Szbarz}
S(z)=\bar z \quad \text{for\,\,} z\in\partial\Omega.
\end{equation}
This can be seen by realizing that the positive part of the power series (\ref{schwarz}) defines the Cauchy transform $C_\Omega(z)$, which can be written, for large $z$,
$$
\sum_{k=0}^{\infty} \frac{M_k}{z^{k+1}}=-\frac{1}{2\pi\I}\int_{\partial\Omega}\frac{\bar\zeta\,d\zeta}{\zeta-z}=C_\Omega (z).
$$
Similarly, the negative part of the series agrees with the power series expansion of a boundary integral which can be interpreted as
the Cauchy transform of the exterior domain:
$$
\sum_{k=-\infty}^{-1} \frac{M_k}{z^{k+1}}=-\frac{1}{2\pi\I}\int_{\partial(\P\setminus\Omega)}\frac{\bar\zeta\,d\zeta}{\zeta-z}
=C_{\P\setminus\Omega}(z)
$$
for $z$ small. Now (\ref{Szbarz}) follows from a well-known jump formula for Cauchy integrals.


\section{Finitely determined domains}\label{sec:finitely}

In order to exhibit transparent links between some of the previously defined objects we shall now work in a basically
algebraic framework. This will mean that we shall work with classes of domains involving only
finitely many parameters, specifically {\it finitely determined domains} in the
terminology of \cite{Putinar-1994}, in other words {\it quadrature domains} \cite{Aharonov-Shapiro-1976},
\cite{Sakai-1982}, \cite{Shapiro-1992}, \cite{Gustafsson-Shapiro-2005}
or {\it algebraic domains} \cite{Varchenko-Etingof-1992}.
In this section our domains need not be simply connected, they will from outset be just bounded
domains in the complex plane (bounded open set would work equally well).

There are many definitions or equivalent characterizations of quadrature domains, the most straightforward
in the context of the present paper perhaps being the following:
a bounded domain $\Omega$ in the complex plane is a (classical) {quadrature domain} (or {algebraic domain}),
if the exterior Cauchy transform is a rational function, {\it i.e.}, if there exists a rational function $R(z)$
such that
\begin{equation}\label{Crational}
C_\Omega(z)=R(z) \quad \text{for all} \,\,z\in \C\setminus\Omega.
\end{equation}

Below is a list of equivalent requirements on a bounded domain $\Omega\subset\C$.
Strictly speaking, the last three items, c)-e), are insensitive for changes of $\Omega$ by
nullsets, but one may achieve equivalence in the pointwise sense by requiring that the
domain $\Omega$ considered is complete with respect to area measure $m$, {\it i.e.}, that all points $a\in\C$
such that $m(\D(a,\varepsilon)\setminus \Omega)=0$ for some $\varepsilon>0$ have been adjoined to $\Omega$.

\begin{itemize}

\item[{a)}]
There exist finitely many points $a_k\in \Omega$ and coefficients $c_{kj}\in \C$ such that
\begin{equation}\label{QD}
\int_\Omega h\, dm =\sum_{k=1}^m\sum_{j=0}^{n_k-1} c_{kj} h^{(j)}(a_k)
\end{equation}
for every function $h$ which is integrable and analytic in $\Omega$.
This is the original definition of quadrature domain, used in \cite{Aharonov-Shapiro-1976} for example.

\item[{b)}] There exists a meromorphic function $S(z)$ in $\Omega$, extending continuously to $\partial\Omega$ with
\begin{equation}\label{Szz}
S(z)=\bar{z} \quad \text{for \,} z\in\partial\Omega.
\end{equation}
This $S(z)$ will be the Schwarz function of $\partial\Omega$ \cite{Davis-1974},
\cite{Shapiro-1992}.

\item[{c)}] The exponential transform $E_{\Omega}(z,w)$ is, for $z, w$ large,
a rational function of the form
\begin{equation}\label{e1}
E_\Omega(z,w)=\frac{Q(z, \bar w)}{P(z)\overline{P(w)}},
\end{equation}
where $P$ and $Q$ are polynomials in one and two variables, respectively.

\item[{d)}]
$\Omega$ is determined by a finite sequence of complex moments $(M_{kj})_{0\leq k,j\leq N}$
(see Example~\ref{ex:determined} below for a clarification of the meaning of this).

\item[{e)}] For some positive integer $N$ there holds
\begin{equation*}\label{b-deg}
\det(B_{kj})_{0\leq k,j\leq N} = 0.
\end{equation*}

\end{itemize}

Basic references for c)-e) are \cite{Putinar-1994}, \cite{Putinar-1995}, \cite{Putinar-1996}, \cite{Gustafsson-Putinar-1998}.

When the above conditions hold then the minimum possible number $N$ in d) and {e}), the degree of $P$ in c) and the number of poles
(counting multiplicities) of $S(z)$ in $\Omega$,
all coincide with the {\it order} of the quadrature domain, {\it i.e.}, the number $N=\sum_{k=1}^{m}n_k$ in (\ref{QD}). For $Q$, see more precisely below.

If $\Omega$ is simply connected, the above conditions a)-e) are also equivalent to that any conformal map $f:\D\to\Omega$ is a rational
function. If $\Omega$ is multiply connected, then it is necessarily finitely connected (in fact, the boundary will, up to finitely many points,
be exactly the algebraic curve $Q(z,\bar z)=0$ with $Q$ as in c) above), and there is a beautiful extension, due to
D. Yakubovich, of the aforementioned statement: $\partial\Omega$ is traced out (on $\partial\D$) by the eigenvalues of a
rational normal matrix function $F:\D\to \C^{N\times N}$ which is holomorphic in $\D$; see \cite{Yakubovich-2006} for details.

The positive definiteness properties of the exponential transform (see end of Section~\ref{sec:Cauchy})
imply that when $\Omega$ is a quadrature domain of order
$N$ then $Q(z,w)$ admits the following representation \cite{Gustafsson-Putinar-2000}:
\begin{equation*}
\begin{split}
Q(z,\bar w)=P(z)\overline{P(w)}-\sum_{k=0}^{N-1}P_k(z)\overline{P_k(w)}.
\end{split}
\end{equation*}
Here each $P_k(z)$ is a polynomial of degree $k$ exactly. Recall that $P(z)=P_N(z)$ has degree $N$, and one
usually normalizes it to be monic (thereby also making $Q(z,\bar w)$ uniquely determined).
Thus the rational form (\ref{e1}) of the exponential transform can be expressed as
$$
1-E_\Omega(z,w)=\sum_{k=0}^{N-1}\frac{P_k(z)\overline{P_k(w)}}{P(z)\overline{P(w)}}.
$$
With the same polynomials, the rational form (\ref{Crational}) of the Cauchy transform is precisely
$$
C_\Omega(z)=\frac{P_{N-1}(z)}{P(z)}.
$$
Note however that the double Cauchy transform $C_\Omega(z,w)$ will not be rational, only its exponential will be.

\begin{example}\label{ex:determined}
The following example should clarify the meaning of a domain $\Omega$ being finitely determined (determined by finitely
many moments $M_{kj}$), as in the equivalent condition d) above. Suppose that, by some measurements, we happen to know
the moments $M_{00}$, $M_{10}$, $M_{01}$, $M_{11}$, but we do not know from which domain they come. Generally speaking,
there will be infinitely many domains whose moment sequence starts out with just these four moments, unless of course the data
are contradictory so that no domain has them. However, for exceptional choices of the data it happens that these four
moments (or, in general, finitely many moments) determine the domain uniquely. This happens exactly for quadrature domains.

Assume for example that
\begin{equation}\label{fourM}
M_{00}=M_{10}=M_{01}=4, \quad M_{11}=12,
\end{equation}
and let us try to compute the exponential moments $B_{kj}$ as far as possible. We have
$$
1-\exp\big[-\sum_{k,j\geq 0} \frac{M_{kj}}{z^{k+1}\bar{w}^{j+1}}\big]
=1-\exp\big[- \frac{{4}}{{z}{\bar{w}}}- \frac{4}{z^{2}\bar{w}}- \frac{4}{z\bar{w}^{2}}- \frac{12}{z^{2}\bar{w}^{2}}-\dots\big]
$$
$$
=\frac{4}{{z}\bar{w}}+ \frac{4}{z^{2}\bar{w}}+ \frac{4}{z\bar{w}^{2}}+ \frac{12}{z^{2}\bar{w}^{2}}+\dots
-\frac{1}{2!}(-\frac{4}{{z}\bar{w}}-\dots)^2-\frac{1}{3!}(-\frac{4}{{z}\bar{w}}-\dots)^3-\dots
$$
$$
=\frac{4}{{z}\bar{w}}+ \frac{4}{z^{2}\bar{w}}+ \frac{4}{z\bar{w}^{2}}+ \frac{4}{z^{2}\bar{w}^{2}}+\text{higher order terms}.
$$
Thus we see that
$$
B_{00}=B_{10}=B_{01}=B_{11}=4.
$$

Clearly the determinant
$\det(B_{kj})_{0\leq k,j\leq N}$ vanishes for $N=1$. The theory (equivalent condition e) above) then
tells that the moments come from a quadrature domain of order one, {\it i.e.}, a disk. To determine which disk one uses
(\ref{oneminusE}) together with (\ref{e1}). Arguing in general one first realizes that
\begin{equation}\label{PPB}
P(z)\overline{P(w)}\sum_{k,j\geq 0} \frac{B_{kj}}{z^{k+1}\bar{w}^{j+1}}=P(z)\overline{P(w)}-Q(z,\bar w),
\end{equation}
in particular the left member is a polynomial. Setting $P(z)=\sum_{j=0}^N \alpha_j z^j$ and chosing $\alpha_N=1$ (normalization),
the vanishing of the coefficient of $1/z\bar w$ in the left member gives
$$
\sum_{k,j=0}^N \alpha_j\bar\alpha_k B_{kj}=0.
$$
This, by the way, explains the condition $\det(B_{kj})_{0\leq k,j\leq N}=0$. It also gives the equation
$$
\sum_{j=0}^N  B_{kj}\alpha_j=0 \quad (k=0,\dots, N)
$$
for the coefficients of $P(z)$.
In our particular example this equation gives $\alpha_0=-\alpha_1=-1$, hence
$$
P(z)=z-1.
$$

Next, $Q(z,\bar w)$ is easily obtained from (\ref{PPB}):
$$
Q(z,\bar w)=P(z)\overline{P(w)}-P(z)\overline{P(w)}\sum_{k,j\geq 0} \frac{B_{kj}}{z^{k+1}\bar{w}^{j+1}}
$$
$$
=(z\bar w-z-\bar w+1)-(z\bar w-z-\bar w+1)(\frac{4}{{z}\bar{w}}+ \frac{4}{z^{2}\bar{w}}+ \frac{4}{z\bar{w}^{2}}+ \frac{4}{z^{2}\bar{w}^{2}}+\dots)
$$
$$
=z\bar w-z-\bar w+1-4 +\text{terms containing negative powers of $z$ or $\bar w$}
$$
$$
=z\bar w-z-\bar w-3= (z-1)(\bar w -1)-4.
$$
Thus we have identified $\Omega$ as the disk
$$
\Omega=\{z\in\C:|z-1|^2<4\}.
$$
No other domain (or open set) has the first four moments given by (\ref{fourM}).

\end{example}


\section{The meromorphic resultant and the elimination function}\label{sec:resultant}

In this section we review the definitions of the meromorphic resultant and the
elimination function, as introduced in \cite{Gustafsson-Tkachev-2009},  referring to that paper
for any details.
If $f$ is a meromorphic function on any compact Riemann surface $M$ we denote by $(f)$ its divisor of zeros and poles,
symbolically $(f)=f^{-1}(0)-f^{-1}(\infty)$ (written in additive form). If $D$ is any divisor and $g$ is any function
we denote by $\scal{D}{g}$ the additive action of $D$ on $g$, and by
$g(D)$ the multiplicative action of $g$ on $D$. See the example below for clarification.

Now the {\it meromorphic resultant} between $f$ and $g$ is, by definition,
$$
\Res (f,g)=g((f))=e^{\scal{(f)}{\log g}}
$$
whenever this makes sense (the resultant is undefined if expressions like $0\cdot\infty$, $\frac{0}{0}$, $\frac{\infty}{\infty}$
appear when the expression in the middle member is evaluated).
In the last expression, $\log g$ refers to arbitrarily chosen local branches of the logarithm.
When defined, $\Res (f,g)$ is either a complex number or $\infty$.
As a consequence of the Weil reciprocity theorem \cite{Weil-1979} the resultant is symmetric:
$$
\Res (f,g)=\Res (g,f).
$$

\begin{example}\label{ex:Res}
To illustrate and explain the above definitions we spell out what they look like for the sample divisor
$$
D=1\cdot (a)+1\cdot (b)-2\cdot (c).
$$
Here $a,b,c\in M$ (any compact Riemann surface) and $g$ is a function on $M$ (typically a meromorphic function,
but this is not absolutely necessary as for the definitions). It is possible to think of
$D$ as a $0$-chain and of $g$ as a $0$-form. Then the above definitions amount to
\begin{equation}\label{additive}
\scal{D}{g}=1\cdot g(a)+1\cdot g(b)-2\cdot g(c)  =\int_D g,
\end{equation}
\begin{equation}\label{multiplicative}
g(D)=e^{\scal{D}{\log g}}=\frac{g(a)g(b)}{g(c)^2}.
\end{equation}
We can associate to $D$ the $2$-form current
\begin{equation}\label{muD}
\delta_D dx\wedge dy =\delta_a dx\wedge dy+\delta_b dx\wedge dy-2\delta_c dx\wedge dy,
\end{equation}
where $\delta_a dx\wedge dy$ denotes the Dirac measure (point mass) at $a$ considered as a $2$-form (so that it acts on
functions by integration over $M$). Then (\ref{additive}) can be augmented to
$$
\scal{D}{g}=\int_D g =\int_M g\delta_D dx\wedge dy.
$$

If $f$ is a meromorphic function with divisor $(f)=D$, then the above $g(D)$
(in (\ref{multiplicative})) is also the resultant $\Res (f,g)=g((f))$.
For example, if $M=\P$ and
$$
f(z)=\frac{(z-a)(z-b)}{(z-c)^2},
$$
then
$$
\Res (f,g)=g((f))=\frac{g(a)g(b)}{g(c)^2}.
$$

It is also possible to let functions of two (or more) variables act on divisors. For example, if $g(z,w)$ is such
a function and $D_1$, $D_2$ are two divisors, say $D_1=1\cdot (a)+1\cdot (b)-2\cdot (c)$, $D_2=3\cdot (p)-3\cdot (q)$, then
(by definition)
\begin{equation}\label{double}
g(D_1, D_2)=\exp\scal{D_1\otimes D_2}{\log g}=\frac{g(a,p)^3 g(b,p)^3 g(c,q)^6}{g(a,q)^3 g(b,q)^3 g(c,p)^6}.
\end{equation}

\end{example}

\bigskip

Next the {\it elimination function} between two meromorphic functions, $f$ and $g$, on $M$ is defined by
$$
\mathcal{E}_{f,g}(z,w)=\Res(f-z,g-w),
$$
where $z,w\in\C$ are parameters. It is always a rational function in $z$ and $w$,
more precisely of the form
\begin{equation}\label{QPR}
\mathcal{E}_{f,g}(z,w)=\frac{Q(z,w)}{P(z)R(w)},
\end{equation}
where $Q$, $P$ and $R$ are polynomials,
and it embodies the necessary (since $M$
is compact) polynomial relationship between $f$ and $g$:
$$
\mathcal{E}_{f,g}(f(\zeta),g(\zeta))=0 \quad (\zeta\in M).
$$
There is also an extended elimination function, defined by
$$
\mathcal{E}_{f,g}(z,w; a, b)=\Res(\frac{f-z}{f-a},\frac{g-w}{g-b}).
$$

To relate the elimination function to the exponential transform one
need integral formulas for the elimination function. For this purpose we shall make some
computations within the framework of currents (distributional differential forms) on $M$.
As a first issue, if $f$ is a meromorphic function in $M$ we need to make its logarithm a single-valued
function almost everywhere on $M$. This is done as usual by making suitable `cuts' on $M$, and by choosing
a branch of the logarithm, call it $\Log f$,
on the remaining set $M\setminus (\text{cuts})$. Its differential in the sense of currents will then be of the form
\begin{equation}\label{dLogf}
d\Log f = \frac{df}{f}+\text{distributional contributions on the cuts}.
\end{equation}
If there is a cut along the $x$-axis, for example, the distributional contribution will be $dy$ times a
measure on the cut, that measure being arc length measure times the size of the jump
(an integer multiple of $2\pi\I$) over the cut, between the two branches of the logarithm.

The next issue is the identity
$$
\frac{1}{2\pi\I}d(\frac{df}{f})=\delta_{(f)}dx\wedge dy.
$$
See \cite{Gustafsson-Tkachev-2009} for the simple proof.

Now, proceeding formally we transform the resultant to an integral in the following series of steps:
$$
\Res (f,g)=g((f))=e^{\scal{(f)}{\Log g}}= \exp\int_{(f)}\Log g
$$
$$
=\exp\big[\int_{M}\Log g\,  \delta_{(f)}dx\wedge dy ]=\exp[\frac{1}{2\pi\I}\int_{M}\Log g  \,d(\frac{df}{f}) \big]
$$
$$
=\exp\big[\frac{1}{2\pi\I}\int_{M} \frac{df}{f}\wedge d\Log g \big].
$$
It should be remarked that, in the final integral, the only contributions come from the jumps of $\Log g$
(the last term in (\ref{dLogf}) when stated for $\Log g$) because outside this set of discontinuities the integrand contains $d\zeta\wedge d\zeta=0$ as a factor.
Similarly one gets, for the elimination function,
\begin{equation}\label{elimint}
\mathcal{E}_{f,g}(z,w)= \exp\big[\frac{1}{2\pi \I}\int_M
\frac{df}{f-z}\wedge d\, \Log ({g-w})\big].
\end{equation}

Now we shall apply (\ref{elimint}) in the case that $M$ is the Schottky double of a plane domain $\Omega$,
and connect it to the exponential transform of $\Omega$. So let $\Omega$ be a finitely connected
plane domain with analytic boundary or, more generally, a bordered Riemann surface and let
\begin{equation*}
M=\widehat\Omega=\Omega \cup\partial\Omega\cup\widetilde{\Omega}
\end{equation*}
be the {\it Schottky double} of $\Omega$, {\it i.e.}, the compact Riemann surface
obtained by completing $\Omega$ with a backside with the opposite
conformal structure, the two surfaces glued together along
$\partial\Omega$ (see \cite{Farkas-Kra-1992}, for example). On
$\widehat\Omega$ there is a natural anticonformal involution $J:\widehat\Omega\to \widehat\Omega$
exchanging corresponding points
on $\Omega$ and $\widetilde{\Omega}$ and having $\partial\Omega$ as fixed
points.

Let $f$ and $g$ be two meromorphic functions on $\widehat\Omega$.
Then
\begin{equation*}
f^*= \overline{ (f\circ J}),\quad g^*= \overline{ (g\circ J})
\end{equation*}
are also meromorphic on $\widehat\Omega$.

\begin{theorem}\label{th:eliminationsymmetric}
With $\Omega$, $\widehat{\Omega}$, $f$, $g$  as above, assume in
addition that $f$ has no poles in $\Omega\cup\partial\Omega$ and
that $g$ has no poles in $\widetilde{\Omega}\cup\partial\Omega$. Then, for large $z$, $w$,
\begin{equation*}
\mathcal{E}_{f,g}(z,\bar{w})= \exp\big[\frac{1}{2\pi \I}\int_\Omega
\frac{df}{f-z}\wedge
\frac{d\overline{g^*}}{\overline{g^*}-{\overline{w}}}\big].
\end{equation*}
In particular,
\begin{equation*}
\mathcal{E}_{f,f^*}(z,\bar{w})= \exp\big[\frac{1}{2\pi \I}\int_\Omega
\frac{df}{f-z}\wedge
\frac{d \overline{f}}{\overline{f}-\overline{w}}\big].
\end{equation*}
\end{theorem}

\begin{proof}
For the divisors of $f-z$ and $g-w$ we have, if $z,w$ are large
enough, $\supp(f-z)\subset\widetilde{\Omega}$, $\supp(g-w)\subset{\Omega}$.
Moreover, $\log (g-w)$ has a single-valued branch, $\Log(g-w)$, in
$\widetilde{\Omega}$.
Using that $g=\overline{g^*}$ on $\partial\Omega$ we therefore get, starting with (\ref{elimint}),
\begin{equation*}
\begin{split}
\mathcal{E}_{f,g}(z,\bar{w})
&=\exp\big[\frac{1}{2\pi \I}\int_{\widehat{\Omega}} \frac{df}{f-z}\wedge d\,
\Log (g-\bar{w})\big]\\
&= \exp\big[\frac{1}{2\pi \I}\int_{{\Omega}}
\frac{df}{f-z}\wedge d\, \Log (g-\bar{w})\big]\\
&= \exp\big[-\frac{1}{2\pi \I}\int_{\partial\Omega}
\frac{df}{f-z}\wedge  \Log (g-\bar{w})\big]\\
&= \exp\big[-\frac{1}{2\pi \I}\int_{\partial\Omega}
\frac{df}{f-z}\wedge  \Log (\overline{g^*}-\bar{w})\big]\\
&=\exp\big[\frac{1}{2\pi\I}\int_\Omega \frac{df}{f -z}\wedge
\frac{d\overline{g^*}}{\overline{g^*}-\bar{w}}\big].
\end{split}
\end{equation*}
as claimed.
\end{proof}

We next specialize to the case that $\Omega$ is a quadrature domain.
Let $S(\zeta)$ be the Schwarz function of $\Omega$.
Then the relation \eqref{Szz} can be interpreted as saying that the
pair of functions $S(\zeta)$ and $\bar{\zeta}$ on $\Omega$ combines into a
meromorphic function on the Schottky double
$\widehat\Omega$ of $\Omega$, namely
that function $g$ which equals $S(\zeta)$ on $\Omega$, and equals $\bar{\zeta}$ on $\widetilde{\Omega}$.

The function $f=g^*=\overline{g\circ J}$ is then
represented by the opposite pair: $\zeta$ on $\Omega$, $\overline{S(\zeta)}$ on $\widetilde{\Omega}$.
Therefore, on $\Omega$ we have $f(\zeta)=g^*(\zeta)=\zeta$, and  Theorem~\ref{th:eliminationsymmetric} immediately gives
\begin{corollary}\label{th:m}
For any quadrature domain $\Omega$,
\begin{equation}\label{expelim}
E_\Omega (z,w) = \mathcal{E}_{f,f^*}(z,\bar{w}) \qquad
(|z|,|w|\gg 1),
\end{equation}
where $f$, $f^*$ are the two meromorphic functions on $\widehat{\Omega}$
given on $\Omega$ by $f(\zeta)=\zeta$, $f^*(\zeta)=S(\zeta)$.
\end{corollary}

An alternative way of conceiving the corollary is to think of $f$ being
defined on an independent surface $W$, so that $f:W\to \Omega$ is a
conformal map. Then $\Omega$ is a quadrature domain if and only if
$f$ extends to a meromorphic function of the Schottky double
$\widehat{W}$, and the assertion of the corollary then is that the exponential
transform of $\Omega$ is given by (\ref{expelim}),
with the elimination function in the right member now taken in
$\widehat{W}$.

If $\Omega$ is simply connected we may choose $W=\D$, so that
$\widehat{W}=\mathbb{P}$ with involution $J:\zeta\mapsto
1/\bar{\zeta}$. Then $f:\D\to \Omega$ is a rational function
when (and only when) $\Omega$ is a quadrature domain, hence we
conclude that $E_\Omega (z,w)$ in this case is the elimination
function for two rational functions, $f(\zeta)$ and
$f^*(\zeta)=\overline{f(1/\bar{\zeta}})$.

Finally we want to illustrate the effectiveness of the notation (\ref{double}) by giving a formula for how
the exponential transforms under  rational conformal maps.

\begin{theorem}\label{TH1}
Let $\Omega$ be bounded domain in the complex plane and let $F$ be a
rational function which is bounded and one-to-one on  $\Omega$.
Then for $z,w\in\C{}\setminus\overline{F(\Omega)}$ we have
\begin{equation}\label{www}
E_{F(\Omega)}(z,w)=E_{\Omega}((F-z),(F-w)).
\end{equation}
\end{theorem}

\begin{proof}
We have
$$
E_{F(\Omega)}(z,w)=\exp \big[\frac{1}{2\pi \I}\int\limits_{F(\Omega)}
\frac{d\zeta\wedge d\overline{\zeta} }{(\zeta-z)(\overline{\zeta}- \bar {w})}\big]
$$
$$
=\exp \big[\frac{1}{2\pi \I}\int\limits_{F(\Omega)}
d\log (\zeta-z)\wedge d\overline{\log (\zeta-w)}\big]
$$
$$
=\exp \big[\frac{1}{2\pi \I}\int\limits_{\Omega}
d\log (F(\zeta)-z)\wedge d\overline{\log (F(\zeta)-w)}\big].
$$

Let $D_z=(F-z)$ denote the divisor of $F(\zeta)-z$, and consider it as a function $D_z:\P\to \Z$, namely that function
which at each point evaluates the order of the divisor (hence is zero at all but finitely
many points). Then
$$
d\log (F(\zeta)-z)=\sum_{\alpha\in \mathbb{P}}D_z(\alpha){d\log}({\zeta-\alpha}),
$$
$$
d\,\overline{\log (F(\zeta)-w)}=\sum_{\alpha\in \mathbb{P}}D_w(\beta){d\log}({\bar\zeta-\bar\beta}).
$$

With these observations we can continue the above series of equalities:
$$
E_{F(\Omega)}(z,w)=\exp \big[\frac{1}{2\pi \I}\int\limits_{\Omega}
d\log (F(\zeta)-z)\wedge d\overline{\log (F(\zeta)-w)}\big]
$$
$$
=\exp \big[\frac{1}{2\pi \I}\sum_{\alpha\in \mathbb{P}}\sum_{\beta\in \mathbb{P}}D_z(\alpha)D_w(\beta)\int\limits_{\Omega}
\frac{d\zeta}{\zeta-\alpha}\wedge\frac{d\bar\zeta}{\bar{\zeta}-\bar{\beta}}\big]
$$
$$
=\prod_{\alpha,\beta \in \mathbb{P}}E_\Omega(\alpha,\beta)^{D_z(\alpha) D_w(\beta)}=E_{\Omega}(D_z,D_w),
$$
which is (\ref{www}).

\end{proof}


\section{Potential theoretic remarks}\label{sec:potential}

Recall the definition (\ref{potential}) of logarithmic potential of a signed measure $\mu$ with compact support
in $\C$. If $\mu$ has not zero net mass, then the behavior of the potential at infinity is such that
it automatically puts  an extra point mass at infinity, to the effect that the potential globally (on $\P$) becomes the
potential of a measure (again denoted $\mu$) having vanishing net mass.
If $\mu$ is given by the divisor of a meromorphic (rational) function $f$, {\it i.e.}, if
$$
d\mu=\delta_{(f)}dx\wedge dy
$$
in the notation (\ref{muD}), then
$$
U_\mu =-\log|f|.
$$

The point we wish to make here is that the resultant $\Res(f,g)$ between two rational functions, $f$ and $g$,
has a corresponding interpretation, namely in terms of the mutual energy.
In general, the mutual energy between two signed measure $\mu$ and $\nu$ of zero net mass (on $\P$) is given by
$$
I(\mu,\nu)=\iint\log \frac{1}{|z-\zeta|}\;d\mu(z)d\nu(\zeta)=\int U_\mu\,d\nu=\int U_\nu\,d\mu
$$
$$
=\frac{1}{2\pi}\int d U_\mu \wedge * d U_\nu.
$$
Some care is needed here since the above is not always well-defined (for example if $\mu$ and $\nu$ have
common point masses). But this is just the same situation as  for the resultant.
Now, with $f$ related as above to $\mu$, and similarly $g$ related to $\nu$ by $d\nu=\delta_{(g)}dx\wedge dy$,
then we have the relationship
$$
I(\mu,\nu)=-\log |\Res(f,g)|.
$$

The proof is just a straight-forward computation:
$$
|\Res(f,g)|^2=\exp \big[\scal{(f)}{\log g}+\scal{(f)}{\overline{\log g}}\big]
$$
$$
=\exp\big[{2\scal{(f)}{\log |g|}}\big]=\exp\big[2\int_{(f)}\log |g|\big]
$$
$$
=\exp\big[{-2\int U^\nu \,d\mu}\big]=\exp\big[{-2I(\mu,\nu)}\big].
$$

Similarly, one can relate the discriminant of $f$ with a renormalized self-energy of $\mu$. We refer to \cite{Gustafsson-Tkachev-2009}
for some details concerning this.

All of the above generalizes well from $\P$ to an arbitrary compact Riemann surface $M$. The potential will then
be less explicit, but it follows from classical theory (for example Hodge theory) that given any
signed measure $\mu$ on $M$ with $\int_M d\mu=0$ there is potential $U_\mu$, uniquely defined up to an additive constant, such that,
considering $\mu$ as a $2$-form current,
\begin{equation}\label{ddU}
-d * d U_\mu=2\pi \mu.
\end{equation}

Choosing in particular $d\mu=\delta_a dx\wedge dy-\delta_a dx\wedge dy$ for two points $a,b\in M$ gives a fundamental potential
that we shall denote $V(z)=V(z,w;a,b)$. Besides depending on $a$ and $b$ it depends on a parameter $w$, for normalization.
The potential is characterized by the singularity structure
\begin{equation}\label{eq:V}
V(z)= V(z,w;a,b)= -\log|z-a|+ \log|z-b|+ {\rm harmonic},
\end{equation}
together with the normalization  $V(w,w;a,b)= 0$ of the additive level. It has a few symmetries,
\begin{equation}\label{eq:symm}
V(z,w;a,b)=V(a,b;z,w)=-V(z,w;b,a),
\end{equation}
and the transitivity property
\begin{equation}\label{eq:trans}
V(z,w;a,b)+V(z,w;b,c)=V(z,w;a,c).
\end{equation}

If the Riemann surface happens to be symmetric (like the Schottky double), with involution $J$, then we have in addition
\begin{equation}\label{eq:VVJ}
V(z,w;a,b)=V(J(z),J(w);J(a),J(b)).
\end{equation}
As an application, the ordinary Green's function $g(z,\zeta)$ for a plane domain (or bordered Riemann surface) $\Omega$
can be obtained from the potential $V$ for $M=\widehat{\Omega}$ by
\begin{equation}\label{eq:GVJ}
g(z,\zeta)=\frac{1}{2}V(z,J(z); \zeta, J({\zeta}))= V(z,w; \zeta,
J({\zeta})),
\end{equation}
where $w$ is an arbitrary point on $\partial\Omega$ and where the
second equality follows from (\ref{eq:symm}), (\ref{eq:trans}),
(\ref{eq:VVJ}). Cf. \cite{Fay-1973}, p.~125f.
See also \cite{Gustafsson-Sebbar-2012}
 for some more details.

The potential $V$ appears implicitly or explicitly in classical texts, e.g.,
\cite{Weyl-1964}, \cite{Schiffer-Spencer-1954}.
Here we discuss it
briefly because it connects to the exponential transform. To start, it has a certain
self-reproducing property, expressed in the identity
\begin{equation}\label{eq:VVV}
V(z,w;a,b)=\frac{1}{2\pi}\int_M dV(\cdot, c;z,w)\wedge
^*dV(\cdot,c;a,b),
\end{equation}
valid for arbitrary non-coinciding points $z,w,a,b,c\in M$. Alternatively,
\begin{equation}\label{eq:VVVpartial}
V(z,w;a,b)=\frac{\I}{\pi}\int_M \partial V(\cdot, c;z,w)\wedge\bar\partial V(\cdot,c;a,b),
\end{equation}
where $\partial V=\frac{\partial V}{\partial z}dz$, $\overline{\partial} V=\frac{\partial V}{\partial \bar z}d\bar z$.
The proof of (\ref{eq:VVV}) is straight-forward:
the right member becomes, after partial integration, an application of (\ref{ddU})
and by using the symmetries (\ref{eq:symm}) and (\ref{eq:trans}),
$$
-\frac{1}{2\pi}\int_M V(\cdot, c;z,w)\wedge
d*dV(\cdot,c;a,b)
$$
$$
=V(a,c;z,w)-V(b,c;z,w)=V(z,w;a,b),
$$
which is the left member.

In the case of the Riemann sphere, $V$ relates to the cross-ratio (\ref{cross} by
\begin{equation}\label{eq:Vcross}
V(z,w;a,b)=-\log|(z:w:a:b)| =-\log
\big|\frac{(z-a)(w-b)}{(z-b)(w-a)}\big|.
\end{equation}
Therefore the extended exponential transform for any domain $\Omega\subset\P$ can be conveniently
expressed in terms of $V$ for $\P$, for example
$$
E_\Omega(z,w;a,b)=\exp\big[\frac{2}{\I\pi}\int_\Omega {\partial}{V(\cdot,c;z,a)}\wedge\overline{\partial} V(\cdot,c;w,b)\big],
$$
and for the modulus,
$$
|E_\Omega(z,w;a,b)|=\exp\big[-\frac{1}{\pi}\int_\Omega {d}{V(\cdot,c;z,a)}\wedge *d V(\cdot,c;w,b) \big].
$$

In the case of a compact Riemann surface, $\int_M \partial V(\cdot, c;z,a)\wedge\bar\partial V(\cdot,c;w,b)$
is purely imaginary, by (\ref{eq:VVVpartial}). In particular we get, for $\Omega=M=\P$,
and by using (\ref{eq:VVVpartial})  and  (\ref{eq:Vcross}), that
$$
E_\P(z,w;a,b)=\exp\big[\frac{2}{\I\pi}\int_\P {\partial}{V(\cdot,c;z,a)}\wedge\overline{\partial} V(\cdot,c;w,b)\big]
$$
$$
=\exp\big[-2V(z,a;w,b)\big]=|(z:a:w:b)|^2.
$$
Thus formula (\ref{crossratio}),  which was asserted and proved in another way in Section~\ref{sec:Cauchy},
has been proved again.


\section{Moment coordinates}\label{sec:moment coordinates}

In this section we restrict again to the case that $\Omega\subset\C$ is a simply connected domain with analytic boundary, and with $0\in\Omega$.
In order to keep the presentation reasonably rigorous we shall, much of the time, actually restrict even more, namely by considering
only those domains which are conformal images of the unit disk under  polynomial conformal maps,
\begin{equation}\label{f}
f(\zeta)=\sum_{k=0}^{N} a_k\zeta^{k+1},
\end{equation}
with  $a_0>0$ (normalization) and $N$ fixed. This will be a manifold, call it $\mathcal{M}_{2N+1}$,
of real dimension $2N+1$, in fact it is an open subset of $\R\times \C^N\cong \R^{2N+1}$ with
$a_0\in\R$ and $a_j\in\C$ ($1\leq j\leq N$) as coordinates.

Conformal images of the unit disk under polynomial conformal maps are exactly
the simply connected quadrature domains with the origin as the sole quadrature node, {\it i.e.}, with $m=1$, $a_1=0$ in the notation
of (\ref{QD}), and the order of the quadrature domain then equals the degree of the polynomial ($N+1$ in the above notation).
The class of domains we are considering here is slightly smaller than this class of quadrature domains
because the  assumption of analytic boundary require our polynomials $f$ to be
univalent in a full neighborhood of the closed unit disk. But it has the advantage that $\mathcal{M}_{2N+1}$ will be an ordinary
open manifold (without boundary). For algebraic purposes one could equally well work with the just locally univalent polynomials,
{\it i.e.}, those polynomials $f$ in (\ref{f}) for which $f'\ne 0$ on $\overline{\D}$.

\begin{remark}
The restriction to polynomial $f$ does not change anything in principle, all formulas will look basically
the same as in the general case. In the other extreme, without changing much in practice one may pass
from the infinite dimensional case to the formal level (with no bothering about convergence)
of transfinite functions, see \cite{Gustafsson-Tkachev-2009b}.
\end{remark}

As real coordinates on $\mathcal{M}_{2N+1}$ we may of course use the real and imaginary parts of the coefficients
appearing in (\ref{f}), namely $a_0$, $\re a_j$, $\im a_j$ ($1\leq j\leq N$), but
much of what we are going to discuss concern what happens when we switch to
$M_0$, $\re M_j$, $\im M_j$ ($1\leq j\leq N$) as local coordinates.

The map (or change of coordinates) $(a_0, a_1, a_2,\dots)\mapsto (M_0, M_1, M_2,\dots)$ is explicitly
given by Richardson's formula \cite{Richardson-1972},
\begin{equation}\label{rich}
M_k=\sum (j_0+1) a_{j_0}\cdots a_{j_{k}}\bar{ a}_{j_0+\ldots +j_{k}+k}.
\end{equation}
The summation here goes  over all multi-indices $(j_0,\ldots,j_k)\geq (0,\ldots,0)$ for which
${j_0+\ldots +j_{k}+k}\leq N$. It is convenient to set $a_j=0$ for $j>N$,
and then the upper bound is not needed. (And with $a_j=0$ for $j<0$ the lower bound is
not needed either.) The formula (\ref{rich}) is proved by pulling (\ref{Mkboundary}) back to the unit circle
by the map (\ref{f}) and then using the residue theorem (after having replaced $\bar f$ by $f^*$).

It is clear from (\ref{rich}) that $M_{N+1}=M_{N+2}=\dots =0$ when $a_{N+1}=a_{N+2}=\dots =0$,
and it follows from basic facts about quadrature domains that the converse statement holds as
well (when $f$ is univalent in a full neighborhood of the closed unit disk).
It is virtually impossible to invert (\ref{rich}) in any explicit way. Even for $N=1$ one would need to solve
a third degree algebraic equation in order to express $a_0$, $a_1$ in terms of $M_0$, $M_1$. And the map
$(a_0, a_1, a_2,\dots)\mapsto (M_0, M_1, M_2,\dots)$ is (generally speaking) not one-to-one, as the
example \cite{Sakai-1978} (and several similar examples, cf. \cite{Zalcman-1987}) shows. The author does not know,
however, of any such example in the present setting of polynomial conformal maps. C.~Ullemar \cite{Ullemar-1980}
gives an example of two polynomials of degree three, one univalent and one only locally univalent, which give rise to the
same moments.

\begin{remark}
Even when $M_{N+1}=M_{N+2}=\dots =0$, the negative moments $M_{-1}$, $M_{-2}$, \dots generally make up a full infinite
sequence (of nonzero numbers). This is related to the fact that when $\Omega$ is a quadrature domain, then its complement
is almost never a quadrature domain (one can easily make sense to the notion of an unbounded quadrature domain, or a quadrature
domain in $\P$).
\end{remark}

\bigskip

A deformation of a domain $\Omega$ as above corresponds to a smooth curve $t\mapsto \Omega(t)$ with $\Omega(0)=\Omega$
(equivalently, $t\mapsto f(\zeta,t)$)  in $\mathcal{M}_{2N+1}$,
and its velocity at $t=0$ is a vector in the real tangent space of $\mathcal{M}_{2N+1}$ at the point in question.
This tangent space has, in terms of the coordinates introduced above, two natural bases, namely
$$
\frac{\partial}{\partial a_0}, \quad \frac{\partial}{\partial \re a_j},\quad \frac{\partial}{\partial \im a_j}
$$
$$
\frac{\partial}{\partial M_0}, \quad \frac{\partial}{\partial \re M_j},\quad \frac{\partial}{\partial \im M_j}
$$
($1\leq j\leq N$), respectively.

It is natural to consider also the corresponding complexified tangent space, obtained by allowing complex coefficients
in front of the above basis vectors. For this tangent space we also have the bases
$$
\frac{\partial}{\partial a_0}, \quad \frac{\partial}{\partial  a_j},\quad \frac{\partial}{\partial \bar a_j}
$$
$$
\frac{\partial}{\partial M_0}, \quad \frac{\partial}{\partial  M_j},\quad \frac{\partial}{\partial \overline{M}_j}.
$$
Similar considerations apply for the cotangent spaces, {\it i.e.}, we have the two natural bases, ${d a_0}$, ${d  a_j}$, ${ d\bar a_j}$ and
${d M_0}$, ${d  M_j}$, ${ d\bar M_j}$, etc.
It should be noted, however, that complex tangent vectors such as $\frac{\partial}{\partial  M_j}$ ($j\geq 1$) do not
really correspond to velocity vectors for deformations of domains, only vectors in the real tangent space do.

The dependence of $f$ on the coefficients $a_0,\dots, a_N$ is certainly analytic by (\ref{f}), but as can be understood from
the appearance of conjugations in (\ref{rich}), the dependence on $M_0,\dots, M_N$ is no longer analytic.
Therefore we prefer to write this dependence as
\begin{equation}\label{fM}
f(\zeta)=f(\zeta; \bar{M}_N,\dots,\bar{M}_1, M_0, M_1,\dots, M_N),
\end{equation}
or just $f(\zeta;M)$, with $M$ shorthand for the list of moments:
\begin{equation}\label{defM}
M=(\bar{M}_N,\dots,\bar{M}_1, M_0, M_1,\dots, M_N).
\end{equation}

Despite the implicit nature of the dependence of $f$ on the moments, several interesting general
statements can be made, and we want to highlight a couple of them.
The first is due to O.~Kuznetsova and V.~Tkachev
\cite{Kuznetsova-Tkachev-2004}, \cite{Tkachev-2005}, who were
able to compute the Jacobian determinant for the transition between the above sets of coordinates, thereby confirming a
conjecture of C.~Ullemar \cite{Ullemar-1980}. The beautiful result can be expressed in terms of the ratio
between the volume forms as follows:
\begin{align}\label{jacobian}
&d\bar M_N\wedge\dots d\bar M_1 \wedge dM_0 \wedge d M_1\wedge\dots \wedge d M_N=\\
=&2a_0^{N^2+3N+1}\Res (f', f'^*)\cdot
d\bar a_N\wedge\dots d\bar a_1 \wedge da_0 \wedge d a_1\wedge\dots \wedge d a_N.
\end{align}
Here $\Res (f', f'^*)$ is the meromorphic resultant between $f'$ and $f'^*$,
as discussed in Section~\ref{sec:resultant}. When $f$ is univalent (or locally univalent), then
$\Res (f', f'^*)\ne 0$, hence the coordinate transition is locally one-to-one.

The second general statement is actually a series of results obtained by M.~Mineev-Weinstein, P.~Wiegmann, A.~Zabrodin,
with later contributations also from other authors,
\cite{Wiegmann-Zabrodin-2000}, \cite{Mineev-Zabrodin-2001}, \cite{Marshakov-Wiegmann-Zabrodin-2002},
\cite{Krichever-Marshakov-Zabrodin-2005}.
We shall below, and in the next two sections, discuss some of these results, namely the string equation (\ref{string}),
the prepotential (\ref{prepotential})
and integrability of moments (\ref{M_k}) (we will provide a proof of (\ref{dprepotential})) and, finally,
Hamiltonian formulations of the evolution equations for the mapping functions (\ref{fH}).

The results are partly formulated in terms of a {\it Poisson bracket}, which for any two functions
$f=f(\zeta;M)$, $g=g(\zeta;M)$, analytic for $\zeta$ in a neighborhood of the unit circle and also
depending on the moments $M$, is defined by
\begin{equation}\label{poisson0}
\{f,g\}=\zeta \frac{\partial f}{\partial \zeta}\frac{\partial g}{\partial M_0}
-\zeta \frac{\partial g}{\partial \zeta}\frac{\partial f}{\partial M_0}.
\end{equation}

As a first choice here we take $f$ to be the conformal map (\ref{f}) and $g=f^*$, the reflection of $f$ in the unit circle:
$f^*(\zeta; M)=\overline{f(1/\bar \zeta; M)}$.
Since we are assuming that $\Omega$ has analytic boundary, $f$ and $f^*$ are both analytic in some neighborhood
of $\partial\D$, hence so is $\{f,f^*\}$. Now it turns out that it is much better than so, in fact
$\{f,f^*\}$ is analytic in the whole Riemann sphere, hence is constant. This remarkable fact is sometimes referred
to as the {\it string equation}, which more precisely reads
\begin{equation}\label{string}
\{f,f^*\}=1.
\end{equation}

As for the proof (which follows \cite{Wiegmann-Zabrodin-2000} etc.) one first notices that $f^*=S\circ f$ ($S$
is the Schwarz function), since this holds on $\partial\D$. Since also $S=S(z;M)$ depends on the moments this spells out to
$$
f^*(\zeta; M)=S(f(\zeta; M); M).
$$
Using hence the chain rule when computing $\frac{\partial f^*}{\partial M_0}$ one arrives at
$$
\{f,f^*\}=\zeta \frac{\partial f}{\partial \zeta}\cdot ( \frac{\partial S}{\partial M_0}\circ f).
$$
Now we notice from (\ref{schwarz}) that
$$
\frac{\partial S}{\partial M_0}(z;M)=\frac{1}{ z} + \text{positive powers of }z,
$$
since the coefficients $M_1, M_2, \dots$ of the remaining negative powers of $z$ in (\ref{schwarz}) are independent of $M_0$.
From this one sees that $\{f,f^*\}$ is holomorphic in $\D$. Similarly, it is holomorphic in
$\P\setminus\overline{\D}$, hence it must in fact be constant, and one easily checks that this
constant is one, proving (\ref{string}).

The string equation (\ref{string}) may appear to be quite remarkable identity for conformal maps, but actually it can be identified with an equation which was discovered much earlier, namely the {\it Polubarinova-Galin equation} \cite{Galin-1945}, \cite{Polubarinova-Kochina-1945}, \cite{Vinogradov-Kufarev-1948}, \cite{Gustafsson-Vasiliev-2006}, \cite{Vasiliev-2009} for the evolution of a Hele-Shaw blob with a source at the origin (Laplacian growth), this taken in combination with the Richardson moment conservation law \cite{Richardson-1972} for such a flow. These considerations will be elaborated in the next section.


\section{Domain deformations driven by harmonic gradients}\label{sec:deformation}

Continuing on the theme of the previous section, we now discuss general domain variations from the point of view of Laplacian growth.
We first consider an, in principle arbitrary, deformation $t\mapsto\Omega(t)$ of a simply connected domain $\Omega=\Omega(0)$.
Let $V_n$ denote the speed of $\Gamma(t)=\partial\Omega(t)$ measured in the outward normal direction.
We need to assume that all data are real analytic. This can for example be expressed by saying that the Schwarz function
$S(z,t)$ of $\Gamma(t)$ is real analytic in the parameter $t$ (besides being analytic in $z$, in a neighborhood of $\Gamma(t)$).
Then, by the Cauchy-Kovalevskaya theorem, there exists, for each $t$, a harmonic function  $g(\cdot,t)$ defined in some
neighborhood of $\Gamma(t)$ and satisfying
\begin{align}\label{Vn}
\begin{cases}
g(\cdot,t)&=0 \quad\quad \text{on}\,\, \Gamma(t),\\
-\frac{\partial g(\cdot,t)}{\partial n}&=V_n \quad \text{on}\,\, \Gamma(t).
\end{cases}
\end{align}
This function $g$ can actually also be defined directly. In fact, it is easy to see that (\ref{Vn})
can be expressed as
$$
\frac{\partial S(z,t)}{\partial t}=-4\frac{\partial g(z,t)}{\partial z}
$$
holding on $\Gamma(t)$, and hence identically (both members being analytic in $z$).
To be precise, the latter equation actually contains (\ref{Vn}) only with the first condition weakened to
$g(\cdot,t)$ being constant on $\Gamma(t)$.
From the above one gets the formula
$$
g(z,t)= -\frac{1}{2}\re \int_{z_0(t)}^z \frac{\partial S}{\partial t}(\zeta,t)d\zeta,
$$
where $z_0(t)$ is an arbitrary point on $\Gamma(t)$. Certainly, there will be exist neighborhood of $\Gamma(0)$
in which all the functions $g(\cdot,t)$, for $|t|$ sufficiently small, are defined.

Near $\Gamma(0)$ we have the usual cartesian coordinates $(x,y)$ ($z=x+\I y$), but from $g$ we also get another pair, $(t,\theta)$, namely
$$
\begin{cases}
t=t(z)\quad \text{determined by\,\,} z\in\Gamma_{t(z)},\\
\theta=-*g(z,t(z)).
\end{cases}
$$
Here $*g$ is the conjugate harmonic function of $g$. We shall also write $d\theta =-*dg$, so this star is actually the Hodge star.
Thus $t=t(z)$ is the time when $\Gamma(t)$ arrives at the point $z$, and $\theta$ is a coordinate along each $\Gamma(t)$.
In order for $t(z)$ to be well-defined, at least locally, we need to assume that $V_n\ne 0$.

Fix a value of $t$ and introduce auxiliary coordinates $(n,s)$ in a neighborhood of $\Gamma(t)$ such that,
on $\Gamma(t)$, $\frac{\partial}{\partial s}$ is the tangent vector and $\frac{\partial}{\partial n}$ is the
outward normal vector (thus $s$ is arc length  along  $\Gamma(t)$). Then we have (on $\Gamma(t)$)
$$
*dg=-\frac{\partial g}{\partial s}dn +\frac{\partial g}{\partial n}ds=\frac{\partial g}{\partial n}ds,
$$
and so
$$
d\theta =-*dg(z,t) = V_n ds.
$$
On the other hand (still on $\Gamma(t)$),
$$
\frac{\partial t}{\partial s}=0,\quad \frac{\partial t}{\partial n}=\frac{1}{V_n},
$$
so that
$$
dt=\frac{1}{V_n}dn.
$$
It follows that $dt\wedge d\theta= dn\wedge ds$ on $\Gamma(t)$. But
since $(\frac{\partial}{\partial n},\frac{\partial}{\partial s})$ is just a rotation of
$(\frac{\partial}{\partial x},\frac{\partial}{\partial y})$, we also have $dn\wedge ds= dx\wedge dy$.
This gives $dt\wedge d\theta= dx\wedge dy$ on $\Gamma(t)$, which does not involve the
auxiliary coordinates $(n,s)$. Hence the latter equation actually holds in a full neighborhood of $\Gamma(t)$.
Equivalently,
\begin{equation}\label{dxdydthetadt}
\frac{\partial(x,y)}{\partial(t,\theta)}=1.
\end{equation}

We summarize.

\begin{proposition}\label{prop:string}
Any real analytic deformation $t\mapsto\Gamma(t)$ can be described, locally in regions where $V_n\ne 0$ and
in terms of local coordinates $(t,\theta)$ introduced as above, by the  ``string equation''
\begin{equation}\label{dthetadt}
dt\wedge d\theta= dx\wedge dy.
\end{equation}
\end{proposition}

Compare discussions in, for example, \cite{Khavinson-Mineev-Putinar-2009}.
The reason for calling (\ref{dthetadt}) a string equation will become clear below.

Now we change the point of view.
In the above considerations, the normal velocity  $V_n$ of the boundary was the given data
(essentially arbitrary, except for regularity requirements)
and the harmonic function $g$ was introduced as a secondary
object. {\it Laplacian growth}, on the other hand, starts with a rule for prescribing a harmonic function
$g=g_\Omega$, associated to any given domain $\Omega$, defined in the domain (or at least near the boundary of it)
and vanishing on the boundary. Then one asks the boundary to move with the normal velocity
$$
V_n=-\frac{\partial g}{\partial n}.
$$
Thus $g$ now works as the generating function which forces the motion of $\partial\Omega$, and setting again $\theta =-*g$,
(\ref{dthetadt}) remains valid in regions where $V_n\ne  0$. Typically, $g$ is specified by suitable source terms in $\Omega$
or by boundary conditions on some fixed boundaries.

The classical case is that $0\in\Omega$ and that $g=g_\Omega$ is the Green's function of
$\Omega$ with pole at $z=0$:
\begin{align}\label{green}
\begin{cases}
g(z,t)=-\log|z|+\text{harmonic} \quad &(z\in\Omega),\\
g(z,t)=0\quad &(z\in\partial\Omega).
\end{cases}
\end{align}
When pulled back to the unit disk by $f:\D\to\Omega$ (normalized conformal map) one then gets
$g(f(\zeta),t)=-\log|\zeta|$, and hence
$\theta= -*g(f(\zeta),t)= \arg \zeta$. Thus $\theta$ will simply be the angular variable on the unit
circle. In addition, one sees that $d\theta$ can be interpreted as harmonic measure (on $\partial\D$, and on $\partial\Omega$).
This classical version of Laplacian growth is connected to the harmonic moments in that it is characterized by the moments
$M_k$, $k\geq 1$, being constants of motions (`integrals'), while $M_0$ evolves linearly in time.
Precisely, with $g$ defined by (\ref{green}) we will have (as functions of $t$)
\begin{equation}\label{conserv}
\begin{cases}
M_0=2t +\text{constant},\\
M_k=\text{constant} \quad (k\geq 1).
\end{cases}
\end{equation}
These conservation laws were discovered by S.~Richardson \cite{Richardson-1972}, and they are easy to verify
directly. In fact, one shows that
$$
\frac{d}{dt}\int_{\Omega(t)}h dm =2\pi h(0)
$$
for any function $h$ which is harmonic in a region containing the closure of $\Omega(t)$, and then (\ref{conserv})
is obtained by choosing $h(z)=z^k$ ($k\geq 0$).

Now we can compare (\ref{string}) with (\ref{dxdydthetadt}) or (\ref{dthetadt}):
since $\{f,f^*\}$ is an analytic function, the restriction of (\ref{string}) to $\partial\D$ contains
all information. And on $\partial\D$ we have, on setting $\zeta=e^{\I\theta}$ and writing $f(e^{\I\theta},t)=x+\I y$,
$$
\{f,f^*\}=2\im \big[\frac{\partial f}{\partial \theta}\frac{\partial \bar f}{\partial M_0}\big]=2\frac{\partial(x,y)}{\partial(M_0,\theta)}.
$$
Since $2\frac{\partial }{\partial M_0}=\frac{\partial }{\partial t}$ by (\ref{conserv}) it follows that
(\ref{string}) is the same as (\ref{dxdydthetadt}), or (\ref{dthetadt}).
This was the reason for calling (\ref{dthetadt}) a string equation.

\begin{remark}
A common variant of the above is that $\Omega$ is a domain in the Riemann sphere, $\infty\in\Omega$,
and that $g$ is the Green's function with pole at infinity.
\end{remark}

\begin{remark}
Most of the above considerations also apply in
higher dimensions, the main difference being that, in $n$ real dimensions, $d\theta =-*dg$ will be an $(n-1)$-form
and that (\ref{dthetadt})  will take the form $dt\wedge d\theta= dx_1\wedge\dots\wedge dx_n$ (the volume form).
Another generalization, elliptic growth, is discussed in \cite{Khavinson-Mineev-Putinar-2009}.
\end{remark}

\medskip

Now we shall connect Laplacian growth evolutions to the derivatives
$\frac{\partial }{\partial M_j}$, $\frac{\partial }{\partial \bar M_j}$.
These evolutions will not be monotonic, so we must allow now $V_n$ to change sign, and hence to be zero at some points
of the boundaries.
Let $L$ denote a real-valued differential operator with constant coefficients acting on functions in $\C$, for example
$L=\frac{\partial^k}{\partial x^k}$. Then, given $\Omega(0)$
with analytic boundary, there exists locally an evolution $t\mapsto \Omega(t)$ such that
\begin{equation}\label{LLG}
\frac{d}{dt } \int_{\Omega(t)} h \,dm=2\pi (Lh)(0)
\end{equation}
for every function $h$ which is harmonic in a region containing the closure of $\Omega(t)$. This is a version of Laplacian
growth, in fact it is the evolution is driven by the harmonic function
$$
g_L(z)=\{L_a g(z,a)\}_{a=0},
$$
where $g(z,a)=-\log|z-a|+\text{harmonic}$ \quad is the ordinary Green's function with pole at $a$, and the differential operator
$L_a=L$ acts on that variable (the location of the pole). Thus, $g_L$ is harmonic in $\Omega$, has a multipole singularity
at the origin and vanishes on $\partial\Omega$.

It is actually not easy to prove that evolutions $\Omega(t)$ driven by a given harmonic gradient really exist, but it can be
done under the right assumptions (analyticity of the initial boundary is usually necessary, for example). An example of such
a proof is given in \cite{Reissig-1993}. Other approaches are presented in
\cite{Tian-1996} and \cite{Lin-2009b}.
In the algebraic framework of the previous section, {\it i.e.}, with
$\Omega\in\mathcal{M}_{2N+1}$, the full evolution will stay in $\mathcal{M}_{2N+1}$ provided
the order $k$ of the differential operator $L$ is at most $N$, and the
existence proof is much easier, see \cite{Gustafsson-1984}. In fact, the existence and uniqueness can even be read off from
(\ref{jacobian}).

For the rest of this section, and also for the next section, even if it does not change anything in practice
it is useful to think of being in the finite dimensional setting of $\mathcal{M}_{2N+1}$,
hence assuming that $k\leq N$.

To prove the assertion that $g_L$ really achieves (\ref{LLG}), recall that the evolution  $t\mapsto\Omega(t)$ is defined by the
outward normal velocity of the boundary being $V_n=-\frac{\partial g_L(z)}{\partial n}$.
Since $-\Delta g(\cdot,a)=2\pi \delta_a$ we have
$$
-\Delta g_L=-\{\Delta L_a g(\cdot,a)\}_{a=0}=-\{L_a\Delta g(\cdot,a)\}_{a=0}=2\pi L\delta_0.
$$
With $h$ harmonic this gives
$$
\frac{d}{dt } \int_{\Omega(t)} h \,dm= \int_{\partial\Omega(t)}V_n h \,ds=-\int_{\partial\Omega(t)}h\frac{\partial g_L}{\partial n}  \,ds
$$
$$
=-\int_{\Omega(t)}h\Delta{ g_L}  \,dm =2\pi Lh (0),
$$
proving (\ref{LLG}).

Useful particular choices of $L$ above are, for $k\geq 1$,
$$
L_1=\frac{\partial^k}{\partial x^k}, \quad
L_2=\frac{\partial^k}{\partial x^{k-1}\partial y}.
$$
We may allow $h$ in (\ref{LLG}) to be complex-valued, and choosing $h(z)=z^j$ and evaluating at $z=0$ one gets
\begin{equation}\label{Lz}
\{L_1 (z^k)\}_{z=0}= k!, \quad \{L_2 (z^k)\}_{z=0}=\I\cdot k!,
\end{equation}
and
$\{L_i (z^j)\}_{z=0}=0$ ($i=1,2$) in all remaining cases.
Similarly, with $h(z)=\bar z^j$ we have
\begin{equation}\label{Lzbar}
\{L_1 (\bar z^k)\}_{z=0}= k!, \quad \{L_2 (\bar z^k)\}_{z=0}=-\I\cdot k!,
\end{equation}
and
$\{L_i (\bar z^j)\}_{z=0}=0$ in the remaining cases.

In what follows, we will denote the time coordinates for the two evolutions corresponding to $L_1$ and $L_2$ by different letters,
namely $t_1$ and $t_2$, respectively.
It is a consequence of (\ref{LLG}) and (\ref{Lz}), (\ref{Lzbar}) that the evolution $t_1\mapsto\Omega(t_1)$
generated by $g_{L_1}$ has the property that
$$
\frac{d}{dt_1}M_k=\frac{d}{dt_1}\bar M_k=2k!, \quad \frac{d}{dt_1}M_j=\frac{d}{dt_1}\bar M_j=0 \quad (j\ne k).
$$
Similarly, for the evolution generated by $g_{L_2}$:
$$
\frac{d}{dt_2}M_k=-\frac{d}{dt_2}\bar M_k=\I\cdot 2k!, \quad \frac{d}{dt_2}M_j=\frac{d}{dt_2}\bar M_j=0 \quad (j\ne k),
$$
In other words, the tangent vector giving the velocity for the evolution driven by $g_{L_1}$ is
$$
\frac{d}{dt_1}=2k!(\frac{\partial}{\partial M_k}+\frac{\partial}{\partial\bar M_k})=2k!\frac{\partial}{\partial \re{M_k}},
$$
and that for $g_{L_2}$ is
$$
\frac{d}{dt_2}=2\I k!(\frac{\partial}{\partial M_k}-\frac{\partial}{\partial\bar M_k})=2k!\frac{\partial}{\partial \im{M_k}}.
$$

Now we wish to extract $\frac{\partial}{\partial M_k}$ and $\frac{\partial}{\partial \bar M_k}$ from the above
relations. We then see that these partial derivatives will correspond to linear combinations with complex (and non-real) coefficients
of the two different evolutions, or tangent vectors, hence they will not in themselves correspond to evolutions of
domains. Still they make sense, of course, as vectors in the complexified tangent space of $\mathcal{M}_{2N+1}$.
Precisely, we get
$$
\frac{\partial}{\partial M_k}
=\frac{1}{4k!}(\frac{d}{dt_1}-\I \frac{d}{dt_2}),
$$
$$
\frac{\partial}{\partial \bar M_k}
=\frac{1}{4k!}(\frac{d}{dt_1}+\I \frac{d}{dt_2}).
$$

This is what we need in order to prove (\ref{dprepotential}). Using (\ref{prepotential}) we get
$$
\frac{\partial \mathcal{F}(\Omega)}{\partial M_k}
=\frac{\partial }{\partial M_k}
\frac{1}{\pi^2}\int_{\D(0,R)\setminus\Omega}\int_{\D(0,R)\setminus\Omega} \log
\frac{1}{|z-\zeta|} dm(z) dm({\zeta})
$$
$$
=\frac{1}{4\pi^2 k!}(\frac{d}{dt_1}-\I \frac{d}{dt_2})\int_{\D(0,R)\setminus\Omega}\int_{\D(0,R)\setminus\Omega} \log
\frac{1}{|z-\zeta|} dm(z) dm({\zeta}).
$$
When computing the $\frac{d}{dt_1}$, for example, $\Omega$ is to be replaced with the evolution $\Omega(t_1)$, and
we evaluate the derivative at $t_1=0$ ($\Omega(0)=\Omega$).  There are two occurrences of $\Omega$ in the expression
for $\mathcal{F}(\Omega)$, and $\mathcal{F}(\Omega)$  symmetric
in them, so it is enough to differentiate one of the occurrences and then multiply the result by two.
Writing $2\log|z-\zeta|=\log(\zeta-z)+\log(\bar\zeta-\bar z)$, and using (\ref{LLG}), (\ref{Mkboundary}) this gives
$$
\frac{d}{dt_1}\big|_{t_1=0}\int_{\D(0,R)\setminus\Omega(t_1)}\int_{\D(0,R)\setminus\Omega(t_1)} \log
|z-\zeta|\, d m(z) d m({\zeta})
$$
$$
=\int_{\D(0,R)\setminus\Omega}\frac{d}{dt_1}\big|_{t_1=0}\big(\int_{\D(0,R)}-\int_{\Omega(t_1)}
\big(\log(\zeta-z)+\log(\bar\zeta-\bar z)\big)\, d m(z)\big) d m({\zeta})
$$
$$
=2\pi\int_{\D(0,R)\setminus\Omega}\big(\frac{(k-1)!}{\zeta^k}+\frac{(k-1)!}{\bar\zeta^k}\big)\,d m(\zeta)
$$
$$
=\I{\pi (k-1)!}\int_{\partial\Omega}\big(\frac{1}{\zeta^k}+\frac{1}{\bar\zeta^k}\big)\,\bar\zeta d \zeta
={-2\pi^2 k!}\big(\frac{1}{k}M_{-k}+\frac{1}{k}\bar M_{-k}\big).
$$

Similarly, the corresponding derivative with respect to $t_2$ gives
$$
\frac{d}{dt_2}\big|_{t_2=0}\int_{\D(0,R)\setminus\Omega(t_1)}\int_{\D(0,R)\setminus\Omega(t_1)} \log
|z-\zeta|\, d m(z) d m({\zeta})
$$
$$
={-2\I \pi^2 k!}(\frac{1}{k}M_{-k}-\frac{1}{k}\bar M_{-k}).
$$
Hence, in the combination $\frac{d}{dt_1}-\I \frac{d}{dt_2}$, $\bar M_{-k}$ will cancel out, and we get
$$
\frac{\partial \mathcal{F}(\Omega)}{\partial M_k}= \frac{1}{k}M_{-k},
$$
as desired.


\section{Hamiltonians}\label{sec:hamiltonians}

As a final topic, we wish to explain, in our notations and settings, the Hamiltonian descriptions
of the evolution of the mapping function $f$ presented in \cite{Wiegmann-Zabrodin-2000} and related articles.
Recall that we work within $\mathcal{M}_{2N+1}$, and we use $M$ as a shorthand notation for the
moments, which serve as coordinates on $\mathcal{M}_{2N+1}$, see (\ref{defM}).

Let $W=W(z;M)$ be a primitive function of the Schwarz function $S=S(z;M)$.
By (\ref{schwarz}), the power series expansion of $W$ will be
\begin{equation}\label{W}
W(z;M)=-\sum_{k\in\Z\setminus\{0\}} \frac{M_k}{kz^k}+M_0\log z +C(M),
\end{equation}
where $C(M)$ is a constant. Decomposing $W$ into real and imaginary parts, $W=w+\I *w$, the real part can be
made perfectly well-defined. In fact, $S(z)=2\frac{\partial w}{\partial z}$, and  since $S(z)=\bar z$ on
$\partial \Omega$ we see that the real-valued function
$$
u(z)= \frac{1}{4}|z|^2 -\frac{1}{2}w(z)
$$
satisfies $\frac{\partial u}{\partial z}=0$ on $\partial \Omega$, hence $u$ is constant on $\partial \Omega$.
We shall fix the free additive constant in $u$, and hence that in $w$, by requiring that $u=0$ on $\partial \Omega$.
It follows that $u$ then satisfies
$$
\begin{cases}
\Delta u =1 \quad \text{in a neighborhood of}\,\,  \partial \Omega,\\
u=\frac{\partial u}{\partial z}=0 \quad \text{on}\,\,  \partial \Omega.
\end{cases}
$$

Next we note from (\ref{W}) that
$$
\frac{\partial W}{\partial M_0}=\log z+\sum_{k>0} \frac{1}{k}\frac{\partial M_{-k}}{\partial M_0}\,z^k +\frac{\partial C}{\partial M_0}.
$$
Thus, as for the real part, $\frac{\partial w}{\partial M_0}=\log |z|+\text{harmonic}$ \quad  in $\Omega$. On $\partial\Omega$ we have
$\frac{\partial w}{\partial M_0}=-2\frac{\partial u}{\partial M_0}=0$, because $u$ vanishes to the second order on $\partial\Omega$.
Hence it follows that $\frac{\partial w}{\partial M_0}$ is
minus the Green's function $g=g_\Omega$ of $\Omega$ with pole at the origin:
\begin{equation}\label{wg}
\frac{\partial w}{\partial M_0}=-g.
\end{equation}

Let $G=g+\I *g$ be the analytic completion of $g$.
This Green's function $G=G(z;M)$ is (like $g$) a function of $z\in\Omega$ and the moments. By (\ref{wg}),
\begin{equation}\label{WG}
G(z;M)=-\frac{\partial W(z;M)}{\partial M_0}.
\end{equation}
Recall that $f=f(\zeta;M)$ denotes the
conformal map (\ref{f}) from $\D$ to $\Omega$. Define the (complex-valued) {\it Hamiltonian function}
$H_0=H_0(\zeta;M)$ associated to $M_0$ by
\begin{equation}\label{H_0}
H_0(\zeta;M)=G(f(\zeta;M);M)=-\log \zeta,
\end{equation}
the last equality because $G\circ f$ is the Green's function of $\D$. We see that $H_0$ actually is independent of the moments.
Therefore, as $\frac{\partial H_0}{\partial M_0}=0$, $\zeta\frac{\partial H_0}{\partial \zeta}=-1$ by (\ref{H_0}), we trivially arrive at
the evolution equation
\begin{equation}\label{varphiH}
\frac{\partial \varphi}{\partial M_0}=\{\varphi,H_0\},
\end{equation}
valid for any smooth function $\varphi=\varphi(\zeta;M)$.

For the higher order moments, there are corresponding identities, and they are more interesting and more
selective: they hold only for $\varphi=f$. We first define, for $k\geq 1$,
\begin{equation}\label{H_k}
H_k(\zeta;M)=-\frac{\partial W(z;M)}{\partial M_k}, \quad \text{where}\,\, z=f(\zeta;M).
\end{equation}
The above means that the derivative $\frac{\partial }{\partial M_k}$ only acts on the
$M_k$ which appears in $W(z;M)$, not that in $f(\zeta;M)$. Always when we write $z$ in place of $f(\zeta;M)$
it shall have this implication.
Now, $G(f(\zeta;M);M)=-\log \zeta$ gives again, by differentiating  with respect to $\zeta$ and $M_k$,
$$
\frac{\partial G}{\partial z}(z;M)\frac{\partial f}{\partial \zeta}(\zeta;M)=-\frac{1}{\zeta},
$$
$$
\frac{\partial G}{\partial z}(z;M)\frac{\partial f}{\partial M_k}(\zeta;M)+\frac{\partial G}{\partial M_k}(z;M))=0.
$$
Thus, multiplying the latter equation with $\zeta\frac{\partial f}{\partial \zeta}$ and using the first equation gives
\begin{equation}\label{fG}
\frac{\partial f}{\partial M_k}(\zeta;M)=\zeta \frac{\partial f}{\partial \zeta}(\zeta;M) \frac{\partial G}{\partial M_k}(z;M).
\end{equation}

Here we wish to remove $G$ the from the right member, in favor of $H_k$. From (\ref{H_k}) we get
$$
\zeta \frac{\partial H_k}{\partial \zeta}(\zeta;M)
=-\frac{\partial^2 W}{\partial z\partial M_k}(z;M)\cdot \zeta \frac{\partial f}{\partial \zeta}(\zeta;M)
$$
and, using also (\ref{WG}),
$$
\zeta \frac{\partial f}{\partial \zeta}(\zeta;M)\frac{\partial H_k}{\partial M_0}(\zeta;M)
$$
$$
=-\zeta \frac{\partial f}{\partial \zeta}(\zeta;M)\frac{\partial^2 W}{\partial z\partial M_k}(z;M)\frac{\partial f}{\partial M_0}(\zeta;M)
-\zeta \frac{\partial f}{\partial \zeta}(\zeta;M)\frac{\partial^2 W}{\partial M_0\partial M_k}(z;M)
$$
$$
=\zeta\frac{\partial H_k}{\partial \zeta}(\zeta;M)\frac{\partial f}{\partial M_0}(\zeta;M)
+\zeta \frac{\partial f}{\partial \zeta}(\zeta;M)\frac{\partial G}{\partial M_k}(z;M).
$$
The last term coincides with the right member in  (\ref{fG}). Thus substituting into (\ref{fG}) we get
\begin{equation}\label{fH}
\frac{\partial f}{\partial M_k}=\{f,H_k\},
\end{equation}
which is the evolution equation for $f$ we wanted to arrive at.

As for the classical Laplacian growth evolution, even though the situation is essentially
trivial in the moment coordinates (in view of the explicit solution (\ref{conserv})),
it may be nice to put it all in a traditional Hamiltonian framework (with real-valued Hamiltonian function).
One possibility then is to choose the phase space to be $\partial\D\times\mathcal{M}_{2N+1}\subset\partial\D\times \R^{2N+1}$,
with real coordinates $(\theta, M_0, \re M_1,\im M_1\dots,\re M_N, \im M_N)$, symplectic form
\begin{align*}
\omega=&\frac{1}{2}d\theta\wedge dM_0 + d\re M_1\wedge d\im M_1+\dots+d\re M_N\wedge d\im M_N \\
=&\frac{1}{2\I}\big[\I d\theta\wedge dM_0 +d\bar M_1\wedge dM_1+ \dots +d\bar M_N\wedge dM_N \big]
\end{align*}
and Hamiltonian function
$$
H(\theta,M_0, \re M_1,\dots, \im M_N )=\theta.
$$

The Hamilton equations are, generally speaking (see \cite{Arnold-1978}),
\begin{equation}\label{hamilton}
-dH = \omega (\xi,\cdot),
\end{equation}
where $\xi=\frac{d}{dt}$ is the velocity vector for the evolution, a vector in the tangent space of the phase space.
In our case we have (using dot for time derivative)
$$
\xi=\dot{\theta}\frac{\partial}{\partial \theta}+\dot{M}_0\frac{\partial}{\partial M_0}+\re\dot{M}_1\frac{\partial}{\partial \re M_1}+
\dots +\im\dot{M}_N\frac{\partial}{\partial \im M_N},
$$
giving in (\ref{hamilton})
$$
\dot{\theta}=0, \quad  \dot{M}_0=2,
$$
$$
\re \dot{M}_j=\im \dot{M}_j=0 \quad (1\leq j\leq N),
$$
as expected (cf. (\ref{conserv})).

Note that the first term in $\omega$, with $dM_0=2dt$, can be identified with $-dx\wedge dy$, in the notation of Proposition~\ref{prop:string}.



\bibliography{bibliography_gbjorn}

\end{document}